\documentclass{amsart}
\usepackage{amsmath, amsthm, amscd, amssymb,latexsym,enumerate}
\usepackage{url}
\usepackage{array}
\usepackage[all]{xy}
\usepackage{enumitem}
\usepackage{hyperref}

\def\Z{\mathbb{Z}}
\def\Q{\mathbb{Q}}
\def\F{\mathbb{F}}

\def\C{\mathbb{C}}
\def\p{{\mathfrak{p}}}

\def\proj{\mathrm{proj}}

\def\g{{f_0}}
\def\ddd{{\delta}}

\def\ee{{\varepsilon}}
\def\eee{{\tilde{\varepsilon}}}

\def\Spec{\mathrm{Spec}}

\def\sep{\mathrm{sep}}

\def\im{\mathrm{im}}

\def\dim{\mathrm{dim}}
\def\exp{\mathrm{exp}}
\def\lcm{\mathrm{lcm}}
\def\log{\mathrm{log}}

\def\m{{\mathfrak m}}

\def\n{{n}}

\def\F{{\mathbb F}}

\def\isom{\xrightarrow{\sim}}

\numberwithin{equation}{section}

\newtheorem{thm}{Theorem}[section]
\newtheorem{lem}[thm]{Lemma}

\newtheorem{prop}[thm]{Proposition}

\theoremstyle{definition}

\newtheorem{algorithm}[thm]{Algorithm}

\newtheorem{ex}[thm]{Example}

\newtheorem{rem}[thm]{Remark}

\title
[Algorithms for commutative algebras over the rational numbers]
{Algorithms for commutative algebras over the rational numbers}
\author[H.\ W.\ Lenstra, Jr.]{H.\ W.\ Lenstra, Jr.}
\address{Mathematisch Instituut, Universiteit Leiden, The Netherlands}
\email{hwl@math.leidenuniv.nl}
\author[A.\ Silverberg]{A.\ Silverberg}
\address{Department of Mathematics, University of California, Irvine, CA 92697, USA}
\email{asilverb@math.uci.edu}
\subjclass[2010]{Primary: 13E10, Secondary: 13P99, 68W30}
\keywords{finite-dimensional commutative algebras, algorithms, Jordan-Chevalley decomposition, lifting idempotents}
\thanks{Support for the research was provided by the Alfred P.~Sloan Foundation.
}
\begin{document}

\begin{abstract} 
The algebras considered in this paper are commutative rings of which the additive
group is a finite-dimensional vector space over the field of rational numbers.
We present deterministic polynomial-time algorithms that, given such an algebra,
determine its nilradical, all of its prime ideals, as well as the corresponding
localizations and residue class fields, its largest separable subalgebra,
and its primitive idempotents.
We also solve the discrete logarithm problem in the multiplicative group
of the algebra.
While deterministic polynomial-time algorithms were known earlier,
our approach is different from previous ones.
One of our tools is a primitive element algorithm; it decides whether
the algebra has a primitive element and, if so, finds one, all in polynomial time.
A methodological novelty is the use of derivations to replace
a Hensel-Newton iteration.
It leads to an explicit formula for lifting idempotents against nilpotents
that is valid in any commutative ring.
\end{abstract}

\maketitle

\section{Introduction}

In the present paper, we mean by a {\em $\Q$-algebra} a {\em commutative} ring
$E$ of which the additive
group is a {\em finite-dimensional} vector space over the field $\Q$ of rational numbers.
We give deterministic polynomial-time algorithms for several basic 
computational questions one may ask about $\Q$-algebras,
taking an approach that differs from previous ones.
The reader wishing to implement our algorithms is warned that
we have attempted to optimize the efficiency of our proofs
rather than of our algorithms.

In algorithms, we specify a $\Q$-algebra $E$ by listing a
system of ``structure constants" $a_{ijk}\in\Q$, for $i,j,k\in\{ 1,2,\ldots,n\}$,
where $n = \dim_\Q(E)$. 
These determine the multiplication, in the sense that for some
$\Q$-basis $e_1,e_2,\ldots,e_n$ of $E$ 
one has $e_ie_j = \sum_{k=1}^n a_{ijk}e_k$ for all $i,j$.
Elements of $E$ are then represented by their vector of coordinates
on that basis, $\Q$-algebra homomorphisms are represented by matrices,
and ideals and subalgebras of $E$ by $\Q$-bases for them,
expressed in $e_1,e_2,\ldots,e_n$.

\subsection*{The Jordan-Chevalley decomposition}
Let $E$ be a $\Q$-algebra. We call $\alpha\in E$ {\em nilpotent} if there is a
positive integer $r$ with $\alpha^r=0$.
The set of nilpotent elements of $E$ is the {\em nilradical}
of $E$, denoted by $\sqrt{0}$ or  $\sqrt{0_E}$. 
We call a polynomial $f\in\Q[X]$ {\em separable}
if $f$ is coprime to its derivative $f'$, and $\alpha\in E$ is called 
{\em separable} (over $\Q$) if there exists a separable polynomial
$f\in\Q[X]$ with $f(\alpha)=0$.
We write $E_\sep$ for the set of separable elements of $E$.

The following result is well known.

\begin{thm}
\label{Ewellknownthm0i}
Let $E$ be a $\Q$-algebra as defined above. Then
 the nilradical $\sqrt{0}$ is an ideal of $E$, and 
 $E_\sep$ is a sub-$\Q$-algebra of $E$.
 Also, one has 
 $$E = E_\sep \oplus \sqrt{0}$$ in the sense that
 the map  $E_\sep \oplus \sqrt{0} \to E$,
 $(u,v)\mapsto u+v$ is an isomorphism of $\Q$-vector spaces.
\end{thm}

Writing an element of $E$ as the sum of a separable element and a 
nilpotent element may be viewed as its {\em Jordan-Chevalley decomposition\/}
or {\em Dunford 
decomposition}. We give a proof of Theorem \ref{Ewellknownthm0i} in
section \ref{sepalgssect}, using a Hensel-Newton iteration.

In section \ref{splitnilradsect} we exhibit a polynomial-time
algorithm for achieving the Jordan-Chevalley decomposition
(Algorithm \ref{xyzalgor}).
Interestingly, it gives rise to a method for determining the
nilradical that differs from the traditional way
of doing this, which depends on the trace function
(see section 4.4 of \cite{FR85}).
And more interestingly, we compute the subalgebra $E_\sep$
at the same time.
Altogether, we obtain the following result, which will
be used in several later algorithms.

\begin{thm}
\label{Ewellknownthm0ii}
There is a deterministic polynomial-time algorithm that,
given a $\Q$-algebra
$E$, computes a $\Q$-basis for $E_\sep$ and a $\Q$-basis for $\sqrt{0}$,
as well as the matrix describing the map
$E_\sep \oplus \sqrt{0} \to E$ from Theorem \ref{Ewellknownthm0i} and its inverse
(which describes the inverse map $E \to E_\sep \oplus \sqrt{0}$).
\end{thm}

An algorithm as in Theorem \ref{Ewellknownthm0ii} can also be found
in the Appendix to \cite{BBC96}.
It is a Hensel iteration, carefully formulated so as to avoid
coefficient blow-up.
Our own proof of Theorem \ref{Ewellknownthm0ii}, which is given
in section \ref{splitnilradsect}, uses no iteration at all.
It depends on the following result, in which $h'$ denotes
the formal derivative of $h\in\Q[X]$.

\begin{thm}
\label{ggprthmintro}
Let $g\in\Q[X]$ be non-zero, and let $E$ be the 
$\Q$-algebra $\Q[X]/(g)$.
Then there is a well-defined surjective $\Q$-linear map
$$
\delta : \Q[X]/(g) \to \Q[X]/(g,g')
$$ 
sending $h+(g)$ to $h'+(g,g'),$ 
and its kernel equals $E_\sep$.
\end{thm}

For the proof,
see section \ref{thm1678pfsect}.

One may wonder whether the 
method for determining $E_\sep$ that  
is described 
in Theorem \ref{ggprthmintro} 
applies directly to any $\Q$-algebra.
That is, if $E$ is a $\Q$-algebra, with module of
K\"ahler differentials $\Omega_{E/\Q}$ and
universal derivation ${\mathrm{d}} : E \to \Omega_{E/\Q}$,
is $\ker({\mathrm{d}})$ necessarily equal to $E_\sep$?
(See \cite{Eisenbud} 
for the unexplained terms.)
We are grateful to Maarten Derickx for having provided a
counterexample: if $I \subset \Q[X,Y]$ denotes the ideal
$(5X^4+Y^3,3XY^2+4Y^3,Y^5)$ then
$f=X^5+XY^3+Y^4+I \in \Q[X,Y]/I=E$ satisfies
$f\in\ker({\mathrm{d}})$ while $f\notin E_\sep = \Q\cdot 1$.

\subsection*{Lifting idempotents against nilpotents}
As an example, we discuss the algebra
$E = \Q[X]/(g)$,
where $g=X^m(1-X)^n$ with $m,n\in\Z_{> 0}$.
The nilradical $\sqrt{0}$ of $E$ is the ideal generated by
$X(1-X) + (g)$, and by Theorem \ref{Ewellknownthm0i} the 
natural map
$$
E_\sep \to E/\sqrt{0} \isom  \Q[X]/(X(1-X))
$$
is an isomorphism of $\Q$-algebras.
It follows that there is a unique element $y\in E_\sep$
that maps to $X + (X(1-X))$, and that this element satisfies
$y(1-y)=0$ and $E_\sep = \Q\cdot 1 \oplus \Q\cdot y$.
By Theorem \ref{ggprthmintro}, we have $y = f+ (g)$,
where $f\in\Q[X]$ is characterized by
$$
\deg(f) < n+m, \quad f\equiv X \bmod X(1-X), \quad f'\equiv 0 \bmod (g,g'),
$$
where $(g,g') = X^{m-1}(1-X)^{n-1}$.
Note that from $y(1-y)=0$ it follows that
$X^m(1-X)^n$ divides $f(1-f)$, and because the
condition $f\equiv X \bmod X(1-X)$ implies $\gcd(1-f,X)=1$
and $\gcd(f,1-X)=1$, one obtains
$$
f\equiv 0 \bmod X^m, \qquad 1-f\equiv 0 \bmod (1-X)^n.
$$
It turns out that in this case we can give explicit formulas
for the polynomial $f$. 
These are contained in the following
result, which is in fact valid for any ring,
commutative or not, and in which
$m=0$, $n=0$ are also allowed.

\begin{thm}
\label{exintro}
Let $R$ be a  
ring, and $m,\n\in\Z_{\ge 0}$.
Then there is a unique polynomial $f\in R[X]$ satisfying
$$
\deg(f) < m+\n, \quad f \in R[X]\cdot X^m, \quad f \in 1+ R[X]\cdot (1-X)^\n,
$$
and it is given by
$$
f= \sum_{i=m}^{m+\n-1}{{m+\n-1}\choose{i}}X^i(1-X)^{m+\n-1-i}
=  \sum_{i=m}^{m+\n-1}(-1)^{i-m}{{m+\n-1}\choose{i}}{{i-1}\choose{i-m}}X^i.
$$
This polynomial satisfies $f^2\equiv f\bmod R[X]\cdot X^m(1-X)^n$.
\end{thm}

For the proof of Theorem \ref{exintro} 
see section \ref{thm1678pfsect}.
For $m=n=2$ one obtains $f=3X^2-2X^3$, a polynomial that is
well known in this context (cf.\ \cite{BBC96}, section 3.2).

An {\em idempotent} of a commutative ring $R$ is an element
$e\in R$ such that $e^2=e$.
In addition to providing an example illustrating Theorem \ref{ggprthmintro},
Theorem \ref{exintro} gives an explicit formula for
lifting idempotents against nilpotents in commutative rings;
that is, if $R$ is a commutative ring with nilradical $\sqrt{0}$,
and $\alpha+\sqrt{0}$ is an idempotent in the ring $R/\sqrt{0}$, then
there is a unique $y\in \alpha+\sqrt{0}$ that is an idempotent in $R$,
and the following result tells us how to write it down.

\begin{thm}
\label{exthmintro}
Let $R$ be a commutative ring, suppose $m,\n\in\Z_{\ge 0}$, and suppose
$\alpha\in R$ satisfies 
$$
\alpha^m(1-\alpha)^\n=0.
$$
Then there is a unique  $y\in R$ satisfying 
$$
y-\alpha \text{ is nilpotent}, \quad y^2=y,
$$
and it is given by $y=f(\alpha)$, with $f$ as in Theorem \ref{exintro}.
\end{thm}

For the proof, see section \ref{thm1678pfsect}.

In principle one can use Theorems \ref{exintro} and  \ref{exthmintro}
for algorithmic purposes, but we shall not do so in the present paper.

\subsection*{Primitive elements}

Let $E$ be a $\Q$-algebra. 
For $\alpha\in E$, we denote by $\Q[\alpha]$
the subalgebra of $E$ generated by $\alpha$; it is the image of the
ring homomorphism $\Q[X] \to E$ sending $f\in\Q[X]$ to $f(\alpha)\in E$.
We call $\alpha\in E$ a {\em primitive element} for $E$ if $\Q[\alpha]=E$.

\begin{thm}
\label{primeltthminto}
For any $\Q$-algebra $E$, the subalgebra $E_\sep$ has a
primitive element. In addition,
there is a deterministic polynomial-time algorithm (Algorithm \ref{primeltalgor} below)
that, given a $\Q$-algebra $E$, 
produces $\alpha\in E$ with $E_\sep=\Q[\alpha]$.
\end{thm}

For the proof of Theorem \ref{primeltthminto}, including the algorithm,
see section \ref{primeltsnewsect}.
We will find Theorem \ref{primeltthminto} useful in determining
$\Spec(E)$.

General $\Q$-algebras do not need to have primitive elements
(see Example \ref{noprimeltexample}).

\begin{thm}
\label{equivthminto}
\begin{enumerate}[leftmargin=5mm]
\item
For any $\Q$-algebra $E$, the following four statements are equivalent:
\begin{enumerate}
\item
$E$ has a primitive element;
\item
for each $\m\in\Spec(E)$, the $E/\m$-vector space $\sqrt{0}/\m\sqrt{0}$
has dimension at most $1$;
\item
$\sqrt{0}$ is a principal ideal of $E$;
\item
each ideal of $E$ is a principal ideal.
\end{enumerate}
\item
There is a deterministic polynomial-time algorithm 
that, given a $\Q$-algebra $E$, decides whether $E$ has a primitive element,
and if so finds one.
\end{enumerate}
\end{thm}

For the proof of (i)
see section \ref{primeltsnewsect}.
For (ii) and the algorithm see section \ref{thm13pfsect}.

\subsection*{The spectrum and idempotents}

The spectrum $\Spec(R)$ of a commutative ring $R$ is its set of prime ideals.
For $\m\in\Spec(R)$, we write $R_\m$ for the localization of $R$ at $\m$.
A {\em primitive idempotent} of $R$ is an idempotent
$e\neq 0$ of $R$ such that for all idempotents $e'\in R$ we have $e'e\in\{0,e\}$.

The following result assembles basic structural information about $\Q$-algebras.

\begin{thm}
\label{Ewellknownthm}
Let $E$ be a $\Q$-algebra. Then:
\begin{enumerate}
\item
for each $\m\in\Spec(E)$, the local ring $E_\m$ is a $\Q$-algebra
with a nilpotent maximal ideal, its residue class field is $E/\m$,
and $E/\m$ is a field extension of $\Q$ of finite degree;
\item
$\Spec(E)$ is finite, and the natural map $E \to \prod_{\m\in\Spec(E)}E_\m$
is an isomorphism of $\Q$-algebras;
\item
the natural map $E \to \prod_{\m\in\Spec(E)}E/\m$
is surjective, and its kernel equals the nilradical $\sqrt{0}$;
\item
the restriction
of the map in {\rm (iii)} to $E_\sep$ is a $\Q$-algebra isomorphism
$$E_\sep \isom \prod_{\m\in\Spec(E)}E/\m.$$
\end{enumerate}
\end{thm}

Theorem \ref{Ewellknownthm} is by no means new, but
since it can be quickly obtained from Theorem \ref{Ewellknownthm0i},
we include a proof in section \ref{thm12pfsect}.

We next consider an algorithmic counterpart of Theorem \ref{Ewellknownthm}.
Part (iii) of Theorem \ref{Ewellknownthm} readily implies that
determining all $\m\in\Spec(E)$, as well as their residue class
fields $E/\m$ and the maps $E\to E/\m$, is
equivalent to determining all primitive idempotents of the ring
$E/\sqrt{0}$, in a sense that is not hard to make precise.
Likewise, by part (ii) of Theorem \ref{Ewellknownthm},
determining all localizations $E_\m$ as well as the maps
$E\to E_\m$ is equivalent to determining all primitive idempotents
of $E$ itself.

We briefly discuss the literature on these two problems.
An algorithm for finding the primitive idempotents of $E/\sqrt{0}$
was proposed in section V of \cite{GMT89}, but
it does not run in polynomial time (\cite{GMT89}, section VI);
a promised sequel to \cite{GMT89} never appeared,
and no details are available on the method for dealing with the
issue that was suggested on p.~86 in section 4 of \cite{IRS94}.
However, in section 7.2 of \cite{FR85} one does find a
polynomial-time algorithm for finding the primitive idempotents
of $E/\sqrt{0}$.
It depends on a primitive element, as does our own method.

To find the primitive idempotents of $E$, it suffices to lift
those of $E/\sqrt{0}$. In section IV of \cite{GMT89}, 
a method for doing this is proposed, based on
Hensel's lemma.
A proof that this method runs in polynomial time can be found
in section 3.2 of \cite{BBC96}; alternatively, one can apply
Theorem \ref{exthmintro}.
Our own procedure is much simpler: from the
primitive idempotents of $E/\sqrt{0}$ and the isomorphism
$E_\sep \isom E/\sqrt{0}$ one immediately obtains the primitive
idempotents of $E_\sep$, which coincide with those of $E$.

Altogether, we obtain a new proof of the following theorem.

\begin{thm}
\label{idempotalgorthm}
There is a deterministic polynomial-time algorithm that,
given a $\Q$-algebra, computes its primitive idempotents.
\end{thm}

For the proof and the algorithms, see section \ref{thm13pfsect}.
All idempotents are obtained as sums of subsets of the set of 
primitive idempotents, but there may be too many to list in
polynomial time.

In terms of prime ideals, Theorem \ref{idempotalgorthm} reads as follows.

\begin{thm}
\label{Ewellknownalg}
There are deterministic polynomial-time algorithms that,
given a $\Q$-algebra $E$, compute all $\m\in\Spec(E)$, 
the $\Q$-algebras $E_\m$ and $E/\m$ for all $\m\in\Spec(E)$, 
as well as the natural maps 
$$
E \to E_\m, \quad E_\m \to E/\m, \quad E \to E/\m, \quad
E \to \prod_{\m\in\Spec(E)}E_\m,
\quad
E_\sep \to \prod_{\m\in\Spec(E)}E/\m,$$ 
and the inverses of the latter two maps.
\end{thm}

For the proof and the algorithms, see section \ref{thm13pfsect}.

\subsection*{Discrete logarithms}
Theorem \ref{Ewellknownalg} contributes a useful ingredient in an
algorithm for finding roots of unity in orders that the authors
recently developed; see \cite{RoU,LenSil,LwS}.
We also proved that, given a $\Q$-algebra $E$, 
one can find, in polynomial time, generators for the group
$\mu(E)$ of roots of unity in $E$, and we presented a solution to the
discrete logarithm problem in $\mu(E)$; see \cite{RoU}.
The following result states that, in fact, the discrete logarithm problem in 
the full multiplicative group $E^\ast$ of $E$ admits an
efficient solution.

\begin{thm}
\label{Estarthm}
\begin{enumerate}[leftmargin=5mm]
\item
There is a deterministic polynomial-time algorithm 
that, given a $\Q$-algebra $E$ 
and a finite system $S$ of elements of $E$, decides whether all
elements of $S$ belong to $E^\ast$, and if so determines a finite set of
generators for the kernel of the group homomorphism
$$
\Z^S \to E^\ast, \qquad (m_s)_{s\in S} \mapsto \prod_{s\in S}s^{m_s}.$$
\item
There is a deterministic polynomial-time algorithm 
that, given a $\Q$-algebra $E$, a finite system $S$ of elements of $E^\ast$
and $t\in E^\ast$, decides whether $t$ belong to the subgroup
$\langle S\rangle$ of $E^\ast$ generated by $S$, and if so produces 
$(m_s)_{s\in S} \in \Z^S$ with $t=\prod_{s\in S}s^{m_s}$.
\end{enumerate}
\end{thm}

The case of Theorem \ref{Estarthm} in which $E$ is assumed to be a
{\em field} is already quite complicated; it was done by G.~Ge \cite{Ge}.
We prove Theorem \ref{Estarthm} in section \ref{thm14pfsect} by a reduction
to the case $E$ is a field 
as dealt with by Ge.

In \cite{BBC96} one finds a result that is in substance equivalent
to Theorem \ref{Estarthm}, with a proof that proceeds along the same lines.

\bigskip

\noindent{\bf Acknowledgment:} 
We thank the referee for helpful comments, and 
especially for pointing out that a number of results were
known earlier, in \cite{FR85,GMT89,IRS94,BBC96}.

\section{Separable $\Q$-algebras}
\label{sepalgssect}

Let $E$ be a $\Q$-algebra and $\alpha\in E$. Then the map
$\Q[X] \to E$, $f\mapsto f(\alpha)$ is a ring homomorphism,
and it is not injective because $\dim_\Q(E)< \infty$.
Hence the kernel of the map is equal to $(g) = g\cdot\Q[X]$
for a unique monic polynomial $g\in\Q[X]$, which is
called the {\em minimal polynomial} of $\alpha$ (over $\Q$);
note that the image $\Q[\alpha]$ of the map is then isomorphic
to $\Q[X]/(g)$ as a $\Q$-algebra. Let $E_\sep \subset E$
be as defined in the introduction.

\begin{lem}
\label{lemma41}
Let $E$ be a $\Q$-algebra and $\alpha\in E$.
Let $g\in\Q[X]$ be the minimal polynomial of $\alpha$.
Then the following are equivalent:
\begin{enumerate}
\item
$\alpha\in E_\sep$,
\item
$(g,g')=1$,
\item
$g$ is squarefree in $\Q[X]$.
\end{enumerate}
\end{lem}

\begin{proof}
That (ii) implies (i) follows from the definition of $E_\sep$,
since $g(\alpha)=0$.
For (i) $\Rightarrow$ (ii), suppose $\alpha\in E_\sep$.
Let $f\in\Q[X]$ be such that  $(f,f')=1$ and $f(\alpha)=0$.
Then $g$ divides $f$, so $f,f'\in (g,g')$ and therefore $(g,g')=1$.
That (ii) and (iii) are equivalent is a direct consequence of the fact
that $\Q$ has characteristic zero.
\end{proof}

\begin{lem}
\label{lemma42}
If $E$ is a $\Q$-algebra,
then $E_\sep$ is a sub-$\Q$-algebra of $E$.
\end{lem}

\begin{proof}
Let $\alpha,\beta\in E_\sep$ with minimal polynomials $g, h$, respectively.
Then one has $g=g_1g_2\cdots g_t$, where $g_1,g_2,\ldots ,g_t\in\Q[X]$
are distinct monic irreducible polynomials, by Lemma \ref{lemma41}.
By the Chinese remainder theorem one now has
$$
\Q[\alpha] \cong \Q[X]/(g) \cong \prod_{i=1}^t \Q[X]/(g_i) =\prod_{i=1}^t K_i,
$$
where each $K_i$ is a field.
The subalgebra $\Q[\alpha,\beta]\subset E$ generated by $\alpha$ and $\beta$
is the image of the map
$\Q[\alpha][Y] \to E$, $f\mapsto f(\beta)$, of which the kernel contains $h(Y)$.
One has 
$$
\Q[\alpha][Y]/(h(Y)) \cong \prod_{i=1}^t K_i[Y]/(h(Y)),
$$
where each $K_i[Y]/(h(Y))$ is a finite product of fields because
$(h,h')=1$.
Hence $\Q[\alpha][Y]/(h(Y))$ is a product of finitely many fields $L_j$.
Each $z\in \Q[\alpha][Y]/(h(Y))$, with components $z_j\in L_j$, is a zero
of a separable polynomial in $\Q[X]$, namely the least common multiple
of the 
minimal polynomials over $\Q$ of the elements $z_j$.
Since there is a surjective $\Q$-algebra homomorphism
$\Q[\alpha][Y]/(h(Y)) \to \Q[\alpha,\beta]$, it follows that each element of
$\Q[\alpha,\beta]$ is separable over $\Q$.
In particular, one has $\alpha\pm \beta, \alpha\beta\in E_\sep$.
Since one also has $\Q\cdot 1\subset E_\sep$ (since $1$ is separable),
it follows that $E_\sep$ is a sub-$\Q$-algebra of $E$.
\end{proof}

\begin{lem}
\label{Henselthma}
Suppose $E$ is a $\Q$-algebra, $\alpha\in E$,
and $f\in\Q[X]$ is a separable polynomial for which $f(\alpha)$
is nilpotent.
Then there exists $y\in \alpha + \sqrt{0}$ such that $f(y)=0$.
\end{lem}

\begin{proof}
Since $f\in\Q[X]$ is separable, one has 
$\Q[X]\cdot f^m + \Q[X]\cdot f' = \Q[X]$ for all $m\in\Z_{>0}$.
Let now $z\in E$ be such that $f(z)$ is nilpotent.
Then it follows that $f'(z)\in E^\ast$, so we can define
$$
z^\ast = z - f(z)/f'(z)\in E.
$$
Note that $z^\ast \in z+\sqrt{0}$ because $f(z) \in \sqrt{0}$.
One has 
$$
f(X+Y) \in f(X) + f'(X)Y + \Q[X,Y]\cdot Y^2,
$$ 
and
substituting $X=z$, $Y= - f(z)/f'(z)$ one finds
$f(z^\ast) \in E\cdot f(z)^2$.
Hence $f(z^\ast)$ is nilpotent as well, so that one also has
$f'(z^\ast)\in E^\ast$. 
It follows, starting from $z=\alpha_0=\alpha$ and iterating the map
$z\mapsto z^\ast$, that there is a well-defined sequence
$\alpha_0,\alpha_1,\alpha_2,\ldots$ defined by
$\alpha_{i+1} = \alpha_i^\ast$.
By the above results, all $\alpha_i$ belong to $\alpha+\sqrt{0}$,
and $f(\alpha_i) \in E\cdot f(\alpha)^{2^i}$.
Thus, for $i$ large enough one has $f(\alpha_i)=0$, and one can
take $y=\alpha_i$.
\end{proof}

\medskip

\noindent{\bf{Proof of Theorem \ref{Ewellknownthm0i}.}}
Since $E$ is commutative, the sum of any two nilpotent elements of $E$ is
nilpotent, and it follows that $\sqrt{0}$ is an ideal of $E$;
in particular, it is a $\Q$-vector space.
Lemma \ref{lemma42} shows that  
$E_\sep$ is a sub-$\Q$-algebra of $E$.
We next prove that  the map
$$
E_\sep\oplus\sqrt{0} \to E, \qquad (u,v)\mapsto u+v
$$ 
is surjective.
Let $\alpha\in E$ have minimal polynomial $g$, and let
$f$ be the product of the distinct monic irreducible factors of $g$.
Then $f$ is separable and $f(\alpha)$ is nilpotent, so
Lemma \ref{Henselthma} shows that we can write $\alpha=u+v$ with
$f(u)=0$ and $v\in\sqrt{0}$.
Then $u\in E_\sep$, and surjectivity follows.
To prove injectivity, it suffices to show 
$E_\sep\cap\sqrt{0} = \{ 0\}$.
Let $\alpha\in E_\sep\cap\sqrt{0}$.
Then the minimal polynomial $g$ of $\alpha$ divides both some separable polynomial
and some polynomial of the form $X^m$ with $m\in\Z_{>0}$, so it
divides $X$, and therefore $\alpha=0$.
\qed

\section{The structure of $\Q$-algebras}
\label{thm12pfsect}

In this section we prove Theorem \ref{Ewellknownthm}.
We begin with a lemma.

\begin{lem}
\label{thm12validfor2}
If the assertions of  Theorem \ref{Ewellknownthm}
are valid for two $\Q$-algebras $E_0$ and $E_1$, then they
are valid for the product algebra $E=E_0\times E_1$.
\end{lem}

\begin{proof}
Since this proof is straightforward, we only indicate the main points.
The two projection maps $E\to E_i$ ($i=0,1$) induce maps
$\Spec(E_i)\to \Spec(E)$, and these allow us to identify
$\Spec(E)$ with the disjoint union of $\Spec(E_0)$ and $\Spec(E_1)$.
If $\m=\m_0 \times E_1 \in \Spec(E)$ comes from 
$\m_0 \in \Spec(E_0)$, then one has
$E/\m \isom E_0/\m_0$ and $E_\m \isom (E_0)_{\m_0}$.
Also, one has $E_\sep = (E_0)_{\sep} \times (E_1)_{\sep}$,
and likewise for the nilradicals.
These facts readily imply the lemma.
\end{proof}

{\bf{Proof of Theorem \ref{Ewellknownthm}.}}
We use induction on $\dim_\Q(E)$.
If $\dim_\Q(E)=0$ then $\Spec(E)=\emptyset$ and all assertions are clear.
Next suppose $\dim_\Q(E)> 0$.

Suppose first that each $\alpha\in E_\sep$ has a minimal polynomial
that is irreducible; then each $\Q[\alpha]$ is a field, so each
non-zero $\alpha\in E_\sep$ has an inverse, so $E_\sep$ is a field.
By $E_\sep\oplus\sqrt{0}=E$ it follows that 
$E/\sqrt{0} \cong E_\sep$, so $\sqrt{0}$ is a maximal ideal.
Each $\p\in\Spec(E)$ contains $\sqrt{0}$, but $\sqrt{0}$ is maximal,
so $\p=\sqrt{0}$.
This proves that $\Spec(E)=\{ \sqrt{0}\}$. 
Hence $E$ is local with maximal ideal $\sqrt{0}$, and
$\sqrt{0}$ is nilpotent because it is a finitely generated
ideal consisting of nilpotents. All statements of
Theorem \ref{Ewellknownthm} follow.

Next assume $\alpha\in E_\sep$ is such that its minimal polynomial $g$ is
reducible: $g=g_0\cdot g_1$ with $g_i\in\Q[X] \smallsetminus \Q$.
Lemma \ref{lemma41} implies that $g_0$ and $g_1$ are coprime, so we have
$$
\Q[\alpha] \cong \Q[X]/(g) \cong \Q[X]/(g_0) \times \Q[X]/(g_1)
$$
as $\Q$-algebras, where both rings $\Q[X]/(g_i)$ are non-zero.
Let $e\in \Q[\alpha]\subset E$ map to
$(1,0)\in \Q[X]/(g_0) \times \Q[X]/(g_1)$.
Then one has $e^2=e$ and $e\notin\{ 0,1\}$.
The ideals $I=Ee$ and $J=E(1-e)$ satisfy $I+J=E$
(because $1\in I+J$) and therefore
$I\cap J = IJ = Ee(1-e) = \{ 0\}$.
The Chinese remainder theorem now shows 
$E = E/(I\cap J) \cong E/I\times E/J = E_0\times E_1$ (say) as
$\Q$-algebras, where $\dim_\Q(E_i) < \dim_\Q(E)$
because $I$ and $J$ are non-zero. The induction hypothesis shows that
Theorem \ref{Ewellknownthm}
is valid for each $E_i$, so by Lemma \ref{thm12validfor2}
it is also valid for $E$.
\qed

\section{Derivations and idempotents}
\label{thm1678pfsect}

In this section we prove Theorems \ref{ggprthmintro},  \ref{exintro}, 
and \ref{exthmintro}.

\medskip

\noindent{\bf{Proof of Theorem \ref{ggprthmintro}.}}
The map $\ddd$ is well-defined since if $g$ divides $h_1-h_2$
then $h_1'-h_2' \in (g,g')$.
It is $\Q$-linear, and it is surjective because the map 
$\Q[X] \to \Q[X]$, $h\to h'$ is surjective.

We prove $E_\sep \subset \ker(\ddd)$.
Let $\alpha\in E_\sep$, and let $f\in\Q[X]$ be a separable polynomial with
$f(\alpha)=0$.
Then $f'(\alpha)\ddd(\alpha) = \ddd(f(\alpha)) = \ddd(0)=0$ and
$f(\alpha)\ddd(\alpha) = 0\cdot \ddd(\alpha) = 0$, and since $(f,f')=1$ one obtains 
$\ddd(\alpha) = 0$.

To prove $E_\sep = \ker(\ddd)$ it will now suffice to prove that
$E_\sep$ and  $\ker(\ddd)$ have the same dimension.
The prime ideals $\m$ of $E$ are in bijective correspondence
with the monic irreducible factors $h$ of $g$ in $\Q[X]$,
by $\m = (h)/(g)$, and $E/\m \cong \Q[X]/(h)$.
The isomorphism $E_\sep \isom \prod_{\m\in\Spec(E)}E/\m$
from Theorem \ref{Ewellknownthm}(iv) now implies
$\dim_\Q(E_\sep) = \sum_h \deg(h) = \deg(\hat{g})$, with
$\hat{g} = \prod_h h$ and $h$ ranging over the monic irreducible
factors of $g$ in $\Q[X]$.
If $h$ occurs exactly $m$ times in $g$, then it
occurs exactly $m-1$ times in $g'$, so 
$(g,g') = (g/\hat{g}).$
Hence 
$$\dim_\Q (\ker(\ddd)) = \deg(g) -\deg(g/\hat{g}) = \deg(\hat{g})
= \dim_\Q (E_\sep)
$$
as required. This proves Theorem \ref{ggprthmintro}.
\qed

\medskip

\noindent{\bf{Proof of Theorem \ref{exintro}.}}
Theorem \ref{exintro} is easy to check if $n=0$, in which case
$f=0$, and also if $m=0$, $n\neq 0$, in which case
$f=1$.
Assume now $m>0$, $n>0$.
We first prove existence.
Define
$$
\g= \sum_{i=0}^{m-1}{{m+\n-1}\choose{i}}X^i(1-X)^{m+\n-1-i},
$$
$$
f= \sum_{i=m}^{m+\n-1}{{m+\n-1}\choose{i}}X^i(1-X)^{m+\n-1-i}.
$$
By the binomial theorem we have
$$
\g+f = (X+(1-X))^{m+\n-1} = 1.
$$
Since one has also $f\in R[X]\cdot X^m$ and $1-f=\g\in R[X]\cdot (1-X)^n$,
this proves existence, and it also proves the first formula for $f$.

To prove uniqueness, suppose that $h\in R[X]$ also satisfies
$$
\deg(h) < m+\n, \quad h \in R[X]\cdot X^m, \quad h \in 1+ R[X]\cdot (1-X)^\n.
$$
Then we have
$
h-f \in R[X]\cdot X^m$  
and $h-f \in R[X]\cdot (1-X)^\n,
$
so
\begin{multline*}
h-f = (h-f)\cdot 1 = (h-f)\g + (h-f)f \\ \in 
R[X]\cdot X^m\cdot R[X]\cdot (1-X)^\n + R[X]\cdot (1-X)^\n\cdot R[X]\cdot X^m \\
= R[X]\cdot X^m\cdot (1-X)^\n,
\end{multline*}
the last equality because $X^m$ and $(1-X)^n$ are central in $R[X]$.
Since $\deg(h-f) < m+n$, it follows that $h=f$. This proves uniqueness.

It remains to prove the second expression for $f$.
First assume that $R=\Q$.
The derivative $f'$ of $f$ is divisible both by $X^{m-1}$ and by
$(1-X)^{n-1}$, so if we write 
$f=\sum_{i=m}^{m+\n-1}c_iX^i$ with $c_i\in\Q$, then one has
$$
\sum_{i=m}^{m+\n-1}ic_iX^{i-1}=qX^{m-1}(1-X)^{n-1},
$$
where $q\in\Q[X]$ is a certain polynomial.
Comparing degrees one finds $q\in\Q$.
The first formula for $f$ shows $c_m={{m+\n-1}\choose{m}}$,
so comparing the terms of degree $m-1$ one finds
$q=mc_m = m{{m+\n-1}\choose{m}}$.
Comparing the terms of degree $i-1$ now yields
$$ic_i = q(-1)^{i-m}{{\n-1}\choose{i-m}} = 
(-1)^{i-m}m{{m+\n-1}\choose{m}}{{n-1}\choose{i-m}}$$
from which it follows that
$$
c_i = (-1)^{i-m}{{m+\n-1}\choose{i}}{{i-1}\choose{i-m}}.
$$
This proves the last equality of Theorem \ref{exintro} in $\Q[X]$.
It is then also valid in $\Z[X]$, and since there is a unique ring homomorphism
$\Z \to R$ it is valid in $R[X]$ as well.

The remainder of $f^2$ upon division by $X^m(1-X)^n$ has the same
properties as $f$, so by uniqueness is equal to $f$.
Hence  $f^2\equiv f\bmod R[X]\cdot X^m(1-X)^n$.
\qed

\medskip

\noindent{\bf{Proof of Theorem \ref{exthmintro}.}}
We first prove that $y=f(\alpha)$ has the properties stated.
If $n=0$ then $y=f(\alpha)=0$, which does satisfy $y^2=y$, and $y-\alpha=-\alpha$
is nilpotent because $\alpha^m=0$.
If $m=0$, $n\neq 0$ then $y=f(\alpha)=1$, which does satisfy $y^2=y$, and $y-\alpha=1-\alpha$
is nilpotent because $(1-\alpha)^n=0$.
For $m\neq 0$, $n\neq 0$, one has 
$y-\alpha=f(\alpha)-\alpha \in R\alpha$ and
$y-\alpha=f(\alpha)-1 + (1-\alpha) \in R(1-\alpha)$,
so
$$
y-\alpha=(y-\alpha)((1-\alpha)+\alpha) \in R\alpha(1-\alpha) + R(1-\alpha)\alpha = R\alpha(1-\alpha).
$$
From $\alpha^m(1-\alpha)^n=0$ and the commutativity of $R$ it now follows
that $y-\alpha$ is nilpotent.
With $\g$ as in the proof of Theorem \ref{exintro}, we have
$$
f(1-f) = f\g \in R[X]\cdot X^m\cdot (1-X)^\n,
$$
so $y-y^2 = f(\alpha)(1-f(\alpha)) \in R\cdot \alpha^m\cdot (1-\alpha)^\n = \{ 0\}$.

To prove uniqueness, suppose that $z\in R$ is such that
$z-\alpha$ is nilpotent and $z^2=z$.
Since $R$ is commutative, we have
$$
(z-y)^3 = z^3 -3z^2y + 3zy^2 -y^3 = z-3zy+3zy-y = z-y
$$
and by induction it follows that $z-y = (z-y)^{3^t}$
for all $t\in\Z_{\ge 0}$.
But since $R$ is commutative, the element
$z-y = (z-\alpha)-(y-\alpha)$ is nilpotent, and choosing $t$
such that $(z-y)^{3^t}=0$ one finds $z=y$.
\qed

\section{Splitting off the nilradical}
\label{splitnilradsect}

In this section we prove Theorem \ref{Ewellknownthm0ii}.
We begin with an algorithm that also determines minimal polynomials.

\begin{algorithm}
\label{xyzalgor}
Given a $\Q$-algebra $E$ and $\alpha\in E$, this algorithm computes
the minimal polynomial of $\alpha$, as well as the unique pair
$(u,v)\in E_\sep \oplus \sqrt{0}$ with $\alpha=u+v$.
\begin{enumerate}
\item
Compute $1,\alpha,\alpha^2,\ldots$ until the smallest $k\in\Z_{\ge 0}$ is found
for which $\alpha^k\in\sum_{i<k}\Q \alpha^i$;
if $\alpha^k = \sum_{i<k}c_i\alpha^i$ with $c_i\in\Q$, output
$g = Y^k - \sum_{i<k}c_iY^i \in \Q[Y]$ as the minimal polynomial of $\alpha$.
\item
Use the Euclidean algorithm to compute $(g,g')$, and compute
$\hat{g} = g/(g,g')$.
\item
Use linear algebra to compute the unique $q\in\Q[Y]$ satisfying
$$
\deg(q) < \deg((g,g')), \qquad q'\hat{g} + q\hat{g}' \equiv 1\bmod (g,g'),
$$
and output $v=q(\alpha)\hat{g}(\alpha)$, $u=\alpha-v$.
\end{enumerate}
\end{algorithm}

\begin{prop}
Algorithm \ref{xyzalgor} is correct and runs in polynomial time.
\end{prop}

\begin{proof}
Step (i) is clearly correct, and it runs in polynomial time
because $k=\dim_\Q(\Q[\alpha]) \le \dim_\Q(E)$.
Step (ii) runs in polynomial time by Chapter 6 of \cite{Gathen}.
Note that $\hat{g}$ is, as in the proof of Theorem~\ref{ggprthmintro},
the product of the distinct monic irreducible factors of $g$
in $\Q[Y]$.
We prove the correctness of step (iii).
Applying Theorem \ref{Ewellknownthm0i} to the
$\Q$-algebra $\Q[\alpha] \cong \Q[Y]/(g)$, we see that there
exists a unique element $v\in \sqrt{0_{\Q[\alpha]}}$ with
$u=\alpha-v\in \Q[\alpha]_\sep$.
Since we have $\sqrt{0_{\Q[\alpha]}} = \Q[\alpha]\hat{g}(\alpha)$, 
and since by Theorem \ref{ggprthmintro} we have
$$\Q[\alpha]_\sep = \{ h(\alpha) : h\in\Q[Y], h'\in (g,g')\},$$
it is equivalent to say that there is a unique element
$v= q(\alpha)\hat{g}(\alpha)$ with 
$$
1-q'\hat{g} - q\hat{g}' \equiv 0\bmod (g,g');
$$
here $q\in\Q[Y]$ is uniquely determined modulo $g/\hat{g} = (g,g')$.
This implies the unique existence of $q$ as in step (iii)
and the correctness of the output.
\end{proof}

\begin{ex}
\label{gfex1}
Let $E=\Q[X]/(X^2+1)^2$.
For all $h\in \Q[X]$, let $\tilde{h}$ denote the image of $h$ in $E$.
We apply Algorithm \ref{xyzalgor} with  $\alpha=\tilde{X}$. 
The minimal polynomial of $\tilde{X}$
is $g=(Y^2+1)^2$,
and we have $\hat{g}=Y^2+1$. 
By Theorem~\ref{ggprthmintro},
$$
E_\sep = \ker(\ddd) = 
\{ \tilde{h} : \deg(h)\le 3 \text{ and } h'\in\Q\cdot(X^2+1)\}
= \Q + \Q\cdot(\tilde{X}^3+3\tilde{X}).
$$
Write $q=a+bY$, substitute into 
$$
q'\hat{g}+q\hat{g}'\equiv 1\bmod(Y^2+1),
$$ 
and solve for
$a,b\in\Q$ to obtain
$q=-\frac{1}{2}Y$ and 
$$
v = q(\tilde{X})\hat{g}(\tilde{X}) = -\frac{1}{2}\tilde{X}(\tilde{X}^2+1) \in \sqrt{0_E}.
$$
Then $u=\tilde{X}-v = 
\frac{1}{2}(\tilde{X}^3 + 3\tilde{X}) \in E_\sep$.
\end{ex}

\begin{ex}
\label{gfex2}
Let $E=\Q[X]/(X^2+1)^3$ and 
let $\tilde{h}$ denote the image in $E$ of $h\in \Q[X]$.
Applying Algorithm \ref{xyzalgor} with  $\alpha=\tilde{X}$, 
we have $g=(Y^2+1)^3$  and $\hat{g}=Y^2+1$.
Then
$$
E_\sep = \ker(\ddd) = \Q\cdot 1+\Q\cdot(3\tilde{X}^5 + 10\tilde{X}^3 + 15\tilde{X}).
$$
Solving for the 4 coefficients of $q$
in 
$$
q'\hat{g}+q\hat{g}'\equiv 1\bmod(Y^4 + 2Y^2+1)
$$ 
and letting $u=\tilde{X}-q(\tilde{X})(\tilde{X}^2+1)$, 
we compute that $q=-\frac{1}{8}(3Y^3+7Y)$ and 
$$
u=\tilde{X}+\frac{1}{8}(3\tilde{X}^2+7)\tilde{X}(\tilde{X}^2+1) =
\frac{1}{8}(3\tilde{X}^5 + 10\tilde{X}^3 + 15\tilde{X})
\in E_\sep.
$$ 
\end{ex}

\begin{algorithm}
\label{splittingalgor}
Given a $\Q$-algebra $E$, this algorithm computes a $\Q$-basis for $E_\sep$
and a $\Q$-basis for $\sqrt{0}$, as well as
the matrices describing the map 
$E_\sep \oplus \sqrt{0} \isom E$, $(u,v) \mapsto u+v$ 
and its inverse.
\begin{enumerate}
\item
Applying Algorithm \ref{xyzalgor} to each of the basis elements
$e_1,e_2,\ldots,e_n$ of $E$, determine elements
$u_1,u_2,\ldots,u_n\in E_\sep$ and
$v_1,v_2,\ldots,v_n\in\sqrt{0}$ such that $e_i=u_i+v_i$ for
$1 \le i \le n$.
\item
Using linear algebra, determine a maximal subset
$I \subset \{ 1,2,\ldots,n\}$ for which $(u_i)_{i\in I}$
is linearly independent over $\Q$, and express each
$u_j$ ($1\le j \le n$) as a $\Q$-linear combination of
$(u_i)_{i\in I}$.
Output $(u_i)_{i\in I}$ as a $\Q$-basis for $E_\sep$.
\item
Using linear algebra, determine a maximal subset
$J \subset \{ 1,2,\ldots,n\}$ for which $(v_i)_{i\in J}$
is linearly independent over $\Q$, and express each
$v_j$ ($1\le j \le n$) as a $\Q$-linear combination of
$(v_i)_{i\in J}$.
Output $(v_i)_{i\in J}$ as a $\Q$-basis for $\sqrt{0}$.
\item
The matrix describing the map $E_\sep \oplus \sqrt{0} \to E$
consists of the coordinates of the vectors $u_i$ ($i\in I$) and
$v_i$ ($i\in J$) when expressed on the basis
$e_1,e_2,\ldots,e_n$, as computed in step (i).
The matrix describing the inverse map $E \to E_\sep \oplus \sqrt{0}$
consists of the coordinates of $u_1,u_2,\ldots,u_n$ on $(u_i)_{i\in I}$ 
as computed in step (ii) and
of the coordinates of $v_1,v_2,\ldots,v_n$ on $(v_i)_{i\in J}$ 
as computed in step (iii).
\end{enumerate}
\end{algorithm}

\medskip

\noindent{\bf{Proof of Theorem \ref{Ewellknownthm0ii}.}}
It is a routine exercise to show that Algorithm \ref{splittingalgor}
has the properties claimed in the statement of Theorem \ref{Ewellknownthm0ii}.
\qed

\begin{rem}
If $0\neq g\in\Q[X]$, and one uses the bases
$\{ X^i\}_{i=0}^{\deg(g)-1}$ for $\Q[X]/(g)$ and 
$\{ X^i\}_{i=0}^{\deg((g,g'))-1}$ for
$\Q[X]/(g,g')$, then one will find that $E_\sep$ has a basis consisting of $1$ and one polynomial of degree $j$ for each $j\in\Z$ with
$\deg((g,g'))< j < \deg(g)$.
\end{rem}

\section{Primitive elements}
\label{primeltsnewsect}

In this section we prove Theorems \ref{primeltthminto}
and \ref{equivthminto}(i).

Suppose $E$ is a $\Q$-algebra.
If $\alpha\in E$, then $\alpha$ is {\em integral} over $\Z$ if
there is exists a monic polynomial $f\in\Z[X]$ such that 
$f(\alpha)=0$.
If $f\in\Z[X]$ is monic and separable, then
$f=\prod_i(X-a_i)$ with $a_i\in\overline{\Z}\subset \C$,
where $\overline{\Z}$ denotes the set of algebraic integers in $\C$, and
we define 
the {\em discriminant} of $f$ to be
$$
\Delta(f)=\prod_{i<j}(a_i-a_j)^2\in\Z \smallsetminus \{ 0\}.
$$

\begin{prop}
\label{primeltprop}
Suppose $E$ is a  
$\Q$-algebra,
suppose $\alpha,\beta\in E$ are separable  over $\Q$
and integral over $\Z$. Let $f\in\Z[X]$ denote the
minimal polynomial of $\alpha$ 
and
let $d\in\Z_{>0}$ be such that $d^2 \nmid\Delta(f)$. 
Then 
$\Q[\alpha,\beta]=\Q[\alpha+d\beta]$.
\end{prop}

\begin{proof}
Let $\Phi$ denote the set of ring homomorphisms
from $\Q[\alpha,\beta]$ to $\C$.
By Theorem \ref{Ewellknownthm0i} we have
$\Q[\alpha,\beta]\subset E_\sep$, so $\Q[\alpha,\beta]=\Q[\alpha,\beta]_\sep$ is the
product of finitely many number fields (Theorem \ref{Ewellknownthm}(iv)).
It follows from this that one has
$\#\Phi = \dim_\Q(\Q[\alpha,\beta])$.

We first show that if $\varphi,\psi\in\Phi$ and 
$\varphi(\alpha+d\beta) = \psi(\alpha+d\beta)$, then 
$\varphi=\psi$.
Suppose that $\varphi,\psi\in\Phi$ and 
$\varphi(\alpha+d\beta) = \psi(\alpha+d\beta)$.
Then $\varphi(\alpha),\psi(\alpha),\varphi(\beta),\psi(\beta)
\in\overline{\Z}$, so
$$
\varphi(\alpha)-\psi(\alpha)=d(\psi(\beta)-\varphi(\beta))
\in d\overline{\Z}.
$$
Write $f=\prod_i(X-a_i)$ with $a_i\in\overline{\Z}\subset \C$.
Suppose $\varphi(\alpha)=a_i$ and $\psi(\alpha)=a_j$.
If $i\neq j$, then
$$
\Delta(f) \in \overline{\Z}(\varphi(\alpha)-\psi(\alpha))^2
\subset \overline{\Z}d^2,
$$ 
so $\Delta(f)/d^2\in \overline{\Z}\cap\Q =\Z$,
contradicting our assumption.
Thus $i=j$, so $\varphi(\alpha)=\psi(\alpha)$, so
$\varphi(\beta)=\psi(\beta)$, so $\varphi=\psi$, proving the claim.

Let $h$ denote the minimal polynomial of $\alpha+d\beta$. 
Then $\Q[\alpha+d\beta] \cong \Q[X]/(h)$, the degree of $h$ is
$\dim_\Q\Q[\alpha+d\beta]$, and $h(\alpha+d\beta)=0$, so
$h(\varphi(\alpha+d\beta))=0$ for all $\varphi\in\Phi$.
Thus, 
\begin{multline*}
\deg h\ge \#\{ z\in\C : h(z)=0\} \ge 
\#\{ \varphi(\alpha+d\beta) : \varphi\in\Phi \}
=\#\Phi  \\
=\dim_\Q \Q[\alpha,\beta]
\ge \dim_\Q \Q[\alpha+d\beta] = \deg h,
\end{multline*}
so all are equal.
It follows that $\Q[\alpha,\beta]=\Q[\alpha+d\beta]$.
\end{proof}

It follows from Algorithm \ref{primeltalgor} below
that finding a primitive element of $E$ can be done in
polynomial time.

\begin{thm}
\label{primeltcor}
Suppose that $E$ is a $\Q$-algebra, 
and suppose
$E_\sep=\Q[\alpha_1,\ldots,\alpha_t]$
with each $\alpha_i \in E$ integral over $\Z$.  
Let $f_i$ denote the minimal polynomial of $\alpha_i$, 
and for $i\in\{1,\ldots,t-1\}$ let $d_i$ be a 
positive integer such that $d_i^2 \nmid \Delta(f_i)$.
Then 
$$
E_\sep=\Q[\alpha_1+d_1\alpha_2+d_1d_2\alpha_3+\cdots +d_1d_2\cdots d_{t-1}\alpha_t].
$$
\end{thm}

\begin{proof}
Each $f_i$ is a monic separable polynomial in $\Z[X]$.
Applying Proposition \ref{primeltprop} gives 
$$
\Q[\alpha_{t-1},\alpha_t]=
\Q[\alpha_{t-1}+d_{t-1}\alpha_t].
$$
Proceeding inductively, we have
\begin{multline*}
E_\sep=\Q[\alpha_1,\ldots,\alpha_{t-2},\alpha_{t-1}+d_{t-1}\alpha_t]
= \\
\Q[\alpha_1,\ldots,\alpha_{t-3},\alpha_{t-2}+d_{t-2}\alpha_{t-1}+d_{t-2}d_{t-1}\alpha_t]
= \cdots   \\
=
\Q[\alpha_1+d_1\alpha_2+d_1d_2\alpha_3+\cdots +d_1d_2\cdots d_{t-1}\alpha_t],
\end{multline*}
as required.
\end{proof}

Theorem \ref{primeltcor} yields
the following deterministic polynomial-time algorithm that proves 
Theorem \ref{primeltthminto}.

\begin{algorithm}
\label{primeltalgor}
Given a $\Q$-algebra $E$,  
this algorithm outputs $\alpha\in E$ such that
$E_\sep=\Q[\alpha]$.
\begin{enumerate}
\item
Applying Algorithms \ref{splittingalgor} and \ref{xyzalgor}, find
a $\Q$-basis $u_1,u_2,\ldots,u_t$ for $E_\sep$ as well as the
minimal polynomial $g_i$ of each $u_i$.
\item
For $i=1,2,\ldots,t$, find a non-zero integer $k_i$ for which
$k_ig_i\in\Z[X]$, compute the minimal polynomial 
$f_i = k_i^{\deg(g_i)}g_i(X/k_i)$ of $k_iu_i$,
as well as its discriminant $\Delta(f_i)$ and the
least positive integer $d_i$ for which
$d_i^2 \nmid \Delta(f_i)$.
\item
With $\alpha_i = k_iu_i$, output
$$
\alpha = \alpha_1 + d_1\alpha_2 + d_1d_2\alpha_3 + \cdots + d_1d_2\cdots d_{t-1}\alpha_t.
$$
\end{enumerate}
\end{algorithm}

\medskip

\noindent{\bf{Proof of Theorem \ref{primeltthminto}.}}
The first assertion of Theorem \ref{primeltthminto} follows from
Theorem \ref{primeltcor}; note that each $\Delta(f_i)$
is non-zero, so that $d_i$ exists.
For the second assertion, we show that
Algorithm \ref{primeltalgor} has the required properties.
Note that each $f_i$ is a monic polynomial in $\Z[X]$,
so the elements $x_i = k_iu_i$ ($1\le i \le t$) form a
$\Q$-basis for $E_\sep$ consisting of elements that are
integral over $\Z$.
Thus the correctness of the algorithm follows from 
Theorem \ref{primeltcor}.
It follows from Corollary 11.19 of \cite{Gathen}
that the computation of $\Delta(f_i)$ can be done
in polynomial time.
We have
$d_i \le {\frac{1}{2}}\log|\Delta(f_i)| + o(1)$ for 
$|\Delta(f_i)| \to\infty$, since
$$|\Delta(f_i)| \ge \lcm(1,2,\ldots,d_i-1)^2 = \exp(2d_i+o(1))$$
as $d_i\to\infty$ (see Chapter XXII of 
\cite{HardyWright}).
\qed

\begin{lem}
\label{primeltcriterion}
Let $E$ be a $\Q$-algebra, and let $\alpha\in E$ be such $E_\sep = \Q[\alpha]$.
Suppose that for each $\m\in\Spec(E)$ there exists
$\varepsilon_\m\in \sqrt{0}/\m\sqrt{0}$ with
$(E/\m)\varepsilon_\m = \sqrt{0}/\m\sqrt{0}$.
Then there exists $\varepsilon \in \sqrt{0}$ with
$E = \Q[\alpha+\ee]$.
\end{lem}

\begin{proof}
Since $E/\sqrt{0} \cong \bigoplus_{\m\in\Spec(E)} E/\m$ as $E$-modules, 
we have
$$
\sqrt{0}/\sqrt{0}^2 = \sqrt{0} \otimes_E (E/\sqrt{0}) \cong
\bigoplus_\m (\sqrt{0} \otimes_E (E/\m)) = 
\bigoplus_\m \sqrt{0}/\m\sqrt{0} = 
\bigoplus_\m (E/\m)\varepsilon_\m
$$
for some $\varepsilon_\m \in \sqrt{0}/\m\sqrt{0}$.
Pick $\eee\in\sqrt{0}/\sqrt{0}^2$ mapping to 
$(\varepsilon_\m)_{\m\in\Spec(E)} \in \bigoplus_\m \sqrt{0}/\m\sqrt{0}$.
Since $(\varepsilon_\m)_{\m\in\Spec(E)}$ generates $\bigoplus_\m \sqrt{0}/\m\sqrt{0}$
as a module over $\prod_\m E/\m \cong E/\sqrt{0}$, it follows that
$\eee$ generates $\sqrt{0}/\sqrt{0}^2$
as a module over $E/\sqrt{0}$.
We have $E = E_{\sep} \oplus \sqrt{0}$ and $E_{\sep} = \Q[\alpha]$
for some $\alpha\in E$. Choose $\ee \in \sqrt{0}$
mapping to $\eee$, so 
$E\cdot\ee + \sqrt{0}^2 = \sqrt{0}$.
Let $f$ be the minimal polynomial of $\alpha$. 
Then $f$ is separable.
Since $f(\alpha)=0$ we have $f(\alpha+\ee) \equiv 0$ mod $(\ee)$,
so there exists $n\in \Z_{>0}$ such that 
$f(\alpha+\ee)^n = 0$.
Since $f$ is separable,
we have $f'(\alpha) \in E^\ast$ 
and $f'(\alpha+\ee) \in E^\ast$.
Also, 
$$
f(\alpha+\ee) \equiv f(\alpha) + \ee f'(\alpha) \bmod \sqrt{0}^2,
$$
so $f(\alpha+\ee) \in E^\ast\cdot\ee + \sqrt{0}^2$.
Thus, $F = \Q[\alpha+\ee]$ is a subring of $E$ mapping
onto $E/\sqrt{0} = \Q[\overline{\alpha}]$ where 
$\overline{\alpha} = \alpha + \sqrt{0} = \alpha+\ee + \sqrt{0}$,
and $\sqrt{0_F} = F \cap \sqrt{0}$ maps onto $\sqrt{0}/\sqrt{0}^2$.
It follows that 
$E = F + \sqrt{0}$
and $\sqrt{0} = \sqrt{0_F}E + \sqrt{0}^2.$

Copying the proof of Lemma 7.4 in Chapter II of \cite{Hartshorne},
we have
$$
\sqrt{0} = \sqrt{0_F}E + \sqrt{0}^2 =  
\sqrt{0_F}E + \sqrt{0_F}\sqrt{0} + \sqrt{0}^3 = 
\sqrt{0_F}E + \sqrt{0}^3 = \ldots = \sqrt{0_F}E 
$$
since $\sqrt{0}$ is nilpotent.
Now 
$$
E = F + \sqrt{0} = F + \sqrt{0_F}E = 
F + \sqrt{0_F}(F + \sqrt{0_F}E) = F + \sqrt{0_F}^2E = \cdots = F
= \Q[\alpha+\ee],
$$
as desired.
\end{proof}

\medskip

\noindent{\bf{Proof of Theorem \ref{equivthminto}(i).}}
We first show (a) $\Rightarrow$ (d).
If $E = \Q[\alpha]$, then $E \cong \Q[X]/(g)$ for some $g\in\Q[X]$, and
since $\Q[X]$ is a principal ideal domain, each ideal of $E$ is principal.

The direction (d) $\Rightarrow$ (c) is obvious.

We now show (c) $\Rightarrow$ (b).
If $\sqrt{0}$ is a principal $E$-ideal, then the $E/\m$-vector space 
$\sqrt{0}\otimes_E E/\m = \sqrt{0}/\m\sqrt{0}$ is generated by a single element,
so it has dimension at most $1$.

Finally, we show (b) $\Rightarrow$ (a).
Let $\varepsilon_\m$ generate the $E/\m$-vector space
$\sqrt{0}/\m\sqrt{0}$, let $\alpha$ be a primitive element for $E_\sep$
(Theorem \ref{primeltthminto}), and apply Lemma \ref{primeltcriterion}.
\qed

\begin{ex}
\label{noprimeltexample}
We give an example of a  $\Q$-algebra $E$ that does
not have a primitive element.
Let $E=\Q[X,Y]/(X^2,XY,Y^2)$.
Then $E = \Q\cdot 1 \oplus \Q\cdot x \oplus \Q\cdot y$
where $x$ and $y$ are the images in $E$ of $X$ and $Y$, respectively.
Then $E_\sep = \Q\cdot 1$ and $\sqrt{0} = \Q\cdot x \oplus \Q\cdot y$.
The unique maximal ideal is $\m=\sqrt{0}$.
We have $\sqrt{0}^2=0$.
Thus, $\sqrt{0}/\m\sqrt{0} = \sqrt{0}/\sqrt{0}^2 = \sqrt{0}$
and $\dim_{E/\m} (\sqrt{0}/\m\sqrt{0}) = 2$.
If $z\in E$, then $z=a+bx+cy$ for some $a,b,c\in\Q$.
Then $(z-a)^2=0$, so $\dim_\Q \Q[z] \le 2 < 3 =\dim_\Q E$.
\end{ex}

\section{Decomposing $\Q$-algebras}
\label{thm13pfsect}

In this section we give the algorithms for Theorems \ref{Ewellknownalg},
 \ref{idempotalgorthm}, 
and \ref{equivthminto}(ii).
The next result (along with Theorems \ref{Ewellknownthm0i}
and  \ref{Ewellknownthm0ii}) shows that to compute
$\Spec(E)$, it suffices to compute $\Spec(E_\sep)$.

\begin{lem}
\label{ipilem}
If $E$ is a $\Q$-algebra, then
the map 
$$
i^\ast : \Spec(E) \to \Spec(E_\sep), \qquad 
\m \mapsto \m \cap E_\sep
$$ 
is bijective.
\end{lem}

\begin{proof}
Let $i : E_\sep \to E$ and $\pi : E \to E/\sqrt{0}$ be the inclusion 
and projection maps, respectively.
The induced map $\pi^\ast : \Spec(E/\sqrt{0}) \to \Spec(E)$
is bijective, since every prime ideal of $E$ contains
$\sqrt{0}$.
The composition $\pi \circ i : E_\sep \to E/\sqrt{0}$ is an isomorphism, since
$E = E_\sep \oplus \sqrt{0}$  as in Theorem \ref{Ewellknownthm0i}.
Thus $i^\ast \circ\pi^\ast$ is bijective.
It follows that $i^\ast$ is bijective.
\end{proof}

\begin{algorithm}
\label{findmalgor}
Given a $\Q$-algebra $E$, the algorithm finds all
$\m\in\Spec(E)$, the fields $E/\m$, 
the $\Q$-algebras $E_\m$,  the primitive idempotents of $E$,
the natural maps 
$$E \to E/\m, \quad
E \to E_\m,  \quad
E_\m \to E/\m,  \quad
E_\sep \to \prod_{\m\in\Spec(E)}E/\m,  \quad
E \to \prod_{\m\in\Spec(E)}E_\m,
$$
and the inverses of the latter two maps.

\begin{enumerate}
\item
Apply Algorithm \ref{primeltalgor}
to produce
$\alpha\in E$ such that $E_\sep = \Q[\alpha]$.
\item
Apply Algorithm \ref{splittingalgor} to obtain a basis for $\sqrt{0}$.
\item
Apply Algorithm \ref{xyzalgor} to 
compute the minimal polynomial $f\in\Q[Y]$ of $\alpha$. 
\item
Use the LLL algorithm \cite{LLL} to factor $f$ into monic irreducible factors in $\Q[Y]$.
\item
For each monic irreducible factor $g$ of $f$ in $\Q[Y]$,
output $\{\alpha^ig(\alpha)\}_{i=0}^{\deg(f/g)-1}$, along with
the basis obtained in step (ii), 
as a basis for a prime ideal $\m\in\Spec(E)$, 
with the elements expressed on the given basis for $E$.
\item
For each $\m\in\Spec(E)$ obtained in step (v),
output $\{\alpha^i\bmod{\m}\}_{i=0}^{\deg(g)-1}$ 
as a basis for $E/\m$, and use linear algebra 
to compute a matrix describing the map $E\to E/\m$.
Then compute the composition $E_\sep \to E \to  \prod_{\m\in\Spec(E)}E/\m$
and invert the matrix for this map 
to produce the inverse map
$\prod_{\m\in\Spec(E)}E/\m 
\to 
E_\sep.$
For each $\m\in\Spec(E)$, compute the image 
$e_\m\in E_\sep$ under the latter map 
of the element that has $1$ in the $\m$-th coordinate
and $0$ everywhere else.
Output $\{ e_\m\}_{\m\in\Spec(E)}$
as the set of primitive idempotents of $E$, 
and output $e_\m E$ as the localization $E_\m$.
The map $E \to E_\m=e_\m E$ is multiplication by $e_\m$,
and this gives the map
$$E \to \prod_{\m\in\Spec(E)}E_\m = \prod_{\m\in\Spec(E)}e_\m E.$$
Its inverse is $(y_\m)_{\m\in\Spec(E)} \mapsto \sum_\m y_\m$.
The map $E_\m \to E/\m$ is $e_\m E \subset E \to E/\m$.
\end{enumerate}
\end{algorithm}

\medskip

\noindent{\bf{Proof of Theorems \ref{Ewellknownalg} and \ref{idempotalgorthm}.}}
The map 
$g \mapsto (g(\alpha))$ is a bijection from the set of
monic irreducible factors of $f$ in $\Q[Y]$ to $\Spec(E_\sep)$.
The set $\{\alpha^ig(\alpha)\}_{i=0}^{\deg(f/g)-1}$ 
in step (v) is a basis for the 
prime ideal $g(\alpha)E_\sep\in\Spec(E_\sep)$.
This basis, along with a basis for $\sqrt{0}$, gives a basis
for the prime ideal 
$$
\m=g(\alpha)E_\sep+\sqrt{0} = (i^\ast)^{-1}(g(\alpha)E_\sep)
$$
of $E$, where $i^\ast$ is the bijection of Lemma~\ref{ipilem}.
Step (v) produces all $\m \in\Spec(E)$ by Lemma~\ref{ipilem}.
Since $Y^2-Y$ is a separable polynomial, the idempotents of $E$
are the same as the idempotents of $E_\sep$.
That $E_\m$ and $e_\m E$ are isomorphic as $\Q$-algebras is seen
as in the proof of Theorem \ref{Ewellknownthm} in section \ref{thm12pfsect}
if one realizes that  $e_\m E$ and $E/(1-e_\m)E$ are isomorphic as $\Q$-algebras.
The correctness of the algorithm now follows, and
it runs in polynomial time since the constituent pieces do.
\qed

\medskip

\noindent{\bf{Proof of Theorem \ref{equivthminto}(ii).}}
We next give the algorithm for Theorem \ref{equivthminto}(ii).

\begin{algorithm}
\label{primeltEalg}
Given a $\Q$-algebra $E$, the algorithm 
decides whether $E$ has a primitive element,
and produces one if it does.
\begin{enumerate}
\item
Apply Algorithms \ref{findmalgor} and \ref{splittingalgor} to compute $\Spec(E)$ and $\sqrt{0}$,
respectively.
\item
For each $\m\in\Spec(E)$, use linear algebra to compute 
$c_\m = \dim_{\Q} (\sqrt{0}/\m\sqrt{0})$ and
$d_\m = \dim_{\Q} (E/\m)$.
\item
If for some $\m\in\Spec(E)$ we have $c_\m > d_\m$, terminate with ``no''.
\item
For each $\m$, if $c_\m = 0$ let $\varepsilon_\m =0$ and otherwise
let $\varepsilon_\m$ be any non-zero element of $\sqrt{0}/\m\sqrt{0}$.
\item
Compute 
$\ee \in \sqrt{0}$
mapping to $(\varepsilon_\m)_{\m}\in \bigoplus_\m (\sqrt{0}/\m\sqrt{0}),$
by  inverting the matrix giving the natural isomorphism
$$
\sqrt{0}/\sqrt{0}^2 \isom \bigoplus_{\m\in\Spec(E)} (\sqrt{0}/\m\sqrt{0}),
$$
using linear algebra.
\item
Apply Algorithm \ref{primeltalgor} to produce
$\alpha\in E_\sep$ such that $E_\sep=\Q[\alpha]$. Output $\alpha+\ee$,
a primitive element for $E$.
\end{enumerate}
\end{algorithm}

Note that
$$
\dim_{E/\m} (\sqrt{0}/\m\sqrt{0}) = 
\frac{\dim_{\Q} (\sqrt{0}/\m\sqrt{0})}{\dim_{\Q} (E/\m)} = {\frac{c_\m}{d_\m}}.
$$
Hence (a) $\Leftrightarrow$ (b) of Theorem \ref{equivthminto}(i) and the construction of
$\varepsilon$ in Lemma \ref{primeltcriterion} prove that  Algorithm \ref{primeltEalg}
is correct. It clearly runs in polynomial time.
This proves Theorem \ref{equivthminto}(ii).

\section{Discrete logarithm algorithm in $\Q$-algebras}
\label{thm14pfsect}

In this section we prove  Theorem \ref{Estarthm}.

Let $E$ be a $\Q$-algebra. We denote the composition of the map
$E \to E_\sep \oplus \sqrt{0}$ from Theorem~\ref{Ewellknownthm0ii}
with the natural projection $E_\sep \oplus \sqrt{0} \to E_\sep$
by $\pi : E \to E_\sep$; in other words, $\pi(\alpha)$ is, for
$\alpha\in E$, the unique element of $E_\sep$ for which
$\alpha-\pi(\alpha)$ is nilpotent.
Equivalently, $\pi$ may be described as the composition of the ring
homomorphism $E \to \prod_{\m\in\Spec(E)}E/\m$ from
Theorem \ref{Ewellknownthm}(iii) with the
inverse of the isomorphism $E_\sep \to \prod_{\m\in\Spec(E)}E/\m$ from
Theorem \ref{Ewellknownthm}(iv). The map $\pi$ is a ring
homomorphism that is the identity on $E_\sep$ and that has kernel
$\sqrt{0}$. Each of Theorem \ref{Ewellknownthm0ii} and Theorem \ref{Ewellknownalg}
shows that there is a polynomial-time algorithm that, given $E$,
produces the matrix describing $\pi$.

\begin{prop}
\label{W1}
Let $E$ be a $\Q$-algebra and let $\pi$ be as just defined.
Then there is a group isomorphism of multiplicative groups
$$
E^* \to  (1+\sqrt{0}) \times \prod_{\m\in\Spec(E)}(E/\m)^*,
\quad 
\alpha \mapsto (\alpha/\pi(\alpha),(\alpha+\m)_{\m\in\Spec(E)}),
$$
and there is a group isomorphism 
$$
\log : 1+\sqrt{0} \to \sqrt{0},
\quad 
1-v \mapsto -\sum_{i=1}^{m-1} v^i/i
$$
from a multiplicative group to an additive group,
where $m\in\Z_{\ge 0}$ is such that $\sqrt{0}^m = 0$.
\end{prop}

\begin{proof}
For each $v\in\sqrt{0}$, one has
$(1-v)^{-1}=\sum_{i=0}^m v^i$ for $m$ sufficiently large,
so $1+\sqrt{0} \subset E^\ast$; it is a subgroup of $E^\ast$
since it is the kernel of the group homomorphism
$E^\ast\to E_\sep^\ast$ induced by $\pi$.
The inclusion map $E_\sep^\ast \subset E^\ast$ provides a
splitting of the short exact sequence
$$
1 \to 1+\sqrt{0} \to E^\ast\to E_\sep^\ast \to 1,
$$
so one has
$E^* \isom   (1+\sqrt{0})  \times E_\sep^\ast$, 
$\alpha \mapsto (\alpha/\pi(\alpha),\pi(\alpha))$.
The isomorphism 
$$
E_\sep \isom \prod_{\m\in\Spec(E)}E/\m
$$
from Theorem \ref {Ewellknownthm}(iv)
induces an isomorphism
$$
E_\sep^\ast \isom \prod_{\m\in\Spec(E)}(E/\m)^\ast
$$
of multiplicative groups. The first isomorphism in
Proposition \ref{W1} follows.

Since $\sqrt{0}$ is a finitely generated ideal consisting of nilpotents, it is
nilpotent itself, so $m$ as in the proposition exists.
One now proves in a routine manner that the map
$\log$ is a group isomorphism, the inverse $\exp$ being given by
$\exp(y)=\sum_{i=0}^{m-1}y^i/i!$.
\end{proof}

\begin{prop}
\label{W2}
There is a deterministic polynomial-time algorithm that, given 
a $\Q$-algebra $E$ that is a field and a finite system $S$ of elements of $E^\ast$, 
determines a $\Z$-basis for the kernel of the group homomorphism
$$
\Z^S \to E^\ast, \qquad (m_s)_{s\in S} \mapsto \prod_{s\in S} s^{m_s}.
$$
\end{prop}

\begin{proof}
See \cite{Ge}, both for the algorithm and the proof.
\end{proof}

\begin{algorithm}
\label{W3}
This algorithm takes as input 
a $\Q$-algebra $E$ and a finite system $S$ of elements of $E$. 
It decides whether one has $S \subset E^\ast$, and if so computes 
a finite set of generators for the kernel of the group homomorphism
$$
\Z^S \to E^\ast, \qquad (m_s)_{s\in S} \mapsto \prod_{s\in S} s^{m_s}.
$$
\end{algorithm}

\begin{enumerate}
\item
Compute all $\m\in\Spec(E)$ and all maps $E\to E/\m$ (Theorem \ref{Ewellknownalg}),
and compute $s+\m \in E/\m$ for all $s\in S$ and $\m\in\Spec(E)$.
If at least one of the elements $s+\m$ is zero, answer ``no'' 
(i.e., $S \not\subset E^\ast$) and terminate.
\item
For each $\m\in\Spec(E)$, determine (using the algorithm in Proposition \ref{W2}) 
a $\Z$-basis for the kernel 
$H_\m$ of the group homomorphism
$$
\Z^S \to (E/\m)^\ast, \qquad (m_s)_{s\in S} \mapsto \prod_{s\in S} (s+\m)^{m_s}.
$$
\item
Find the smallest $m\in\Z_{> 0}$ with $\sqrt{0}^m = 0$, and for
each $s\in S$, compute $\log(s/\pi(s))\in \sqrt{0}$,
using a matrix for $\pi$ and the formula for $\log$ in Proposition \ref{W1}.
\item
Compute a basis for the kernel $H$ of the group homomorphism
$$
\Z^S \to \sqrt{0}, \qquad (m_s)_{s\in S} \mapsto 
\sum_{s\in S} {m_s}\log(s/\pi(s)),
$$ 
by applying the 
kernel algorithm in \S 14 of
\cite{HWLMSRI} to an integer multiple of the rational
matrix describing the map. 
\item
Compute a basis $B$ for $H \cap \bigcap_{\m\in\Spec(E)} H_\m \subset \Z^S$
by applying the kernel algorithm in \S 14 of
\cite{HWLMSRI} to  the group homomorphism
$$H \oplus \bigoplus_{\m\in\Spec(E)} H_\m \to \bigoplus_{\m\in\Spec(E)}\Z^S,
\quad
(h,(h_\m)_{\m\in\Spec(E)})\mapsto (h-h_\m)_{\m\in\Spec(E)}.$$
\item
Output the image of $B$ under the projection from
$H \oplus \bigoplus_{\m\in\Spec(E)} H_\m$ to its $H$-component.
\end{enumerate}

\begin{algorithm}
\label{Estarthm2algor}
The algorithm takes a $\Q$-algebra $E$, a finite system $S$ of elements of $E^\ast$
and $t\in E^\ast$, and decides whether $t$ belong to the subgroup
$\langle S\rangle$ of $E^\ast$ generated by $S$, and if so produces 
$(m_s)_{s\in S} \in \Z^S$ with $t=\prod_{s\in S}s^{m_s}$.
\end{algorithm}

\begin{enumerate}
\item
Apply Algorithm \ref{W3} with $T = \{ t\} \cup S$ in place of $S$
to obtain a finite set of generators $U_T$ for the kernel of the group homomorphism
$$
g_T : \Z^{\{ t\}} \times\Z^S = \Z^T \to E^\ast, \qquad
(m_t,(m_s)_{s\in S}) \mapsto t^{m_t}\prod_{s\in S} s^{m_s}.
$$
\item
Map the elements $u\in U_T \subset \Z^{\{ t\}} \times\Z^S$ to their
$\Z^{\{ t\}}$-components $u(t) \in\Z$.
If $\sum_{u\in U_{T}} u(t)\Z \neq \Z$ then $t\not\in \langle S \rangle$;
if $1=\sum_{u\in U_T}n_u u(t)$ with $(n_u)_{u\in U_T}\in\Z^{U_T}$
then $t\in \langle S \rangle$ and
the $\Z^S$-component $(m_s)_{s\in S}$ of 
$-\sum_{u\in U_T}n_u u \in \Z^T = \Z^{\{t\}} \times \Z^{S}$ 
satisfies $t = \prod_{s\in S}s^{m_s}$.
\end{enumerate}

\medskip

\noindent{\bf{Proof of Theorem \ref{Estarthm}(i).}}
We show that Algorithm \ref{W3} is correct and runs in polynomial time.
In step (i), one has $S \subset E^\ast$ if and only if each $s+\m\neq 0$, because 
in any commutative ring the unit group is the complement of the
union of all prime ideals.
As we saw in the proof of Proposition \ref{W1}, the ideal
$\sqrt{0}$ is nilpotent.
With $m$ as in (iii), one has 
$$
E\supset \sqrt{0} \supsetneq \sqrt{0}^2 \supsetneq \cdots \supsetneq \sqrt{0}^{m-1} \supsetneq \sqrt{0}^m =(0), 
$$
so $m \le \dim_\Q(E)$.
Thus, step (iii) runs in polynomial time.
From the isomorphism in Proposition \ref{W1} it follows that
the kernel of $\Z^S \to E^\ast$ equals the intersection
$H \cap \bigcap_{\m\in\Spec(E)} H_\m$ considered in step (v).
It follows that the algorithm gives the correct output.
\qed

\medskip

\noindent{\bf{Proof of Theorem \ref{Estarthm}(ii).}}
Algorithm \ref{Estarthm2algor} is correct and runs in polynomial time since
\begin{align*}
t \in \langle S \rangle 
& \iff \exists x\in\Z^S  \text{ such that }  (1,-x) \in\ker(g_{T}) = \langle U_{T} \rangle \\
& \iff 1 \in\im(\proj : \langle U_{T} \rangle \subset \Z\times\Z^S\to \Z) \\
& \iff 
\exists (n_u)_{u\in U_{T}}, \exists x\in\Z^S  \text{ such that }  
\sum_{u\in U_{T}} n_u u = (1,-x) \in  \Z^{\{ t\}}\times\Z^S 
\end{align*}
where $\proj$ denotes projection onto the first component.
\qed

\end{document}